\documentclass[11pt]{article}
\usepackage{amsmath,amsthm,amsfonts}
\usepackage{fullpage}
\usepackage{mathpazo}
\usepackage{dsfont}
\usepackage{microtype}

\newcommand{\C}{\mathbb{C}}
\DeclareMathOperator{\tr}{tr}

\DeclareMathAlphabet{\varmathbb}{U}{bbold}{m}{n}
\newcommand{\one}{\mathds{1}}
\DeclareMathOperator*{\Exp}{\mathbb{E}}
\DeclareMathOperator{\End}{End}
\DeclareMathOperator{\rk}{\textbf{rk}}

\begin{document}

\newtheorem{theorem}{Theorem}
\newtheorem{lemma}{Lemma}
\newtheorem{claim}{Claim}
\newtheorem{proposition}{Proposition}
\newtheorem{exercise}{Exercise}
\newtheorem{corollary}{Corollary}
\theoremstyle{definition}
\newtheorem{definition}{Definition}

\newcommand{\eps}{\epsilon}
\newcommand{\Z}{{\mathbb{Z}}}
\newcommand{\Var}{\mathop\mathrm{Var}}
\newcommand{\Cov}{\mathop\mathrm{Cov}}
\renewcommand{\Re}{\mathop\mathrm{Re}}
\newcommand{\frob}[1]{\left\| #1 \right\|_{\rm F}}
\newcommand{\opnorm}[1]{\left\| #1 \right\|_{\rm op}}
\newcommand{\bopnorm}[1]{\bigl\| #1 \bigr\|_{\rm op}}
\newcommand{\norm}[1]{\left\| #1 \right\|}
\newcommand{\bnorm}[1]{\bigl\| #1 \bigr\|}
\newcommand{\abs}[1]{\left| #1 \right|}
\newcommand{\inner}[2]{\left\langle #1, #2 \right\rangle}
\renewcommand{\outer}[2]{\left| #1 \right\rangle \left\langle #2 \right|}

\newcommand{\wf}{\widehat{f}}
\newcommand{\wg}{\widehat{g}}
\newcommand{\wG}{\widehat{G}}
\newcommand{\wH}{\widehat{H}}
\newcommand{\wS}{\widehat{S}}
\newcommand{\wpsi}{\widehat{\psi}}
\newcommand{\wpsidag}{\widehat{\psi^\dagger}}
\newcommand{\U}{\textsf{U}}
\newcommand{\GL}{\textsf{GL}}
\newcommand{\fsub}{\operatorname{F}}
\newcommand{\dmin}{d_{\min}}
\newcommand{\tA}{\tilde{A}}
\newcommand{\tpsi}{\tilde{\psi}}

\title{Approximate Representations \\ and Approximate Homomorphisms}

\author{Cristopher Moore\\Computer Science Department\\University of New Mexico\\and the Santa Fe Institute\\{\texttt{moore@cs.unm.edu}}
\and Alexander Russell\\Computer Science and Engineering\\University of Connecticut\\{\texttt{acr@cse.uconn.edu}}}

\maketitle

\begin{abstract}
  Approximate algebraic structures play a defining role in arithmetic
  combinatorics and have found remarkable applications to basic
  questions in number theory and pseudorandomness.  Here we study
  \emph{approximate representations} of finite groups: functions $\psi
  : G \to \U_d$ such that $\Pr[\psi(xy) = \psi(x) \,\psi(y)]$ is
  large, or more generally $\Exp_{x,y} \norm{\psi(xy) - \psi(x)
    \,\psi(y)}_2^2$ is small, where $x, y$ are uniformly random
  elements of the group $G$ and $\U_d$ denotes the unitary group of
  degree $d$.  We bound these quantities in terms of the ratio $d /
  \dmin$ where $\dmin$ is the dimension of the smallest nontrivial
  representation of $G$.  As an application, we bound the extent to
  which a function $f:G \to H$ can be an approximate homomorphism
  where $H$ is another finite group.  We show that if $H$'s
  representations are significantly smaller than $G$'s, no such $f$
  can be much more homomorphic than a random function.
  
  We interpret these results as showing that if $G$ is quasirandom,
  that is, if $\dmin$ is large, then $G$ cannot be embedded in a small
  number of dimensions, or in a less-quasirandom group, without
  significant distortion of $G$'s multiplicative structure.  We also
  prove that our bounds are tight by showing that minors of genuine
  representations and their polar decompositions are essentially
  optimal approximate representations.
\end{abstract}

In additive combinatorics and number theory, an \emph{approximate
  subgroup} of a group $G$ is a subset $H$ which is roughly closed
under multiplication: that is, such that $\Pr_{x,y} [xy \in H]$ is
large where $x,y$ are uniformly random elements of $H$.
We focus on \emph{approximate group representations}---functions
$\psi$ from $G$ to $\U_d$, the group of $d \times d$ unitary matrices,
such that $\psi$ acts roughly like a homomorphism.  We then use our
results to bound the existence of approximate homomorphisms from $G$
to another finite group $H$.
  
Let $G$ be a finite group and let $\psi : G \to \U_{d_\psi}$.  If $\psi(xy) = \psi(x) \,\psi(y)$ for all $x,y \in G$, then we call $\psi$ a \emph{representation}.  We are interested in understanding how close $\psi$ can be to a representation if $G$ does not in fact have any $d_\psi$-dimensional representations---in particular, in the case where $G$ is \emph{quasirandom}~\cite{gowers} in the sense that its smallest nontrivial representation has dimension $\dmin > d_\psi$.  

We can measure the extent to which $\psi$ fails to act as a
representation by the expected $\ell_2$ distance between $\psi(xy)$
and $\psi(x) \,\psi(y)$, where $x$ and $y$ are chosen uniformly from
$G$.  To control the trivial case where $\psi(x) = \one$ for all $x$,
we assume that $\Exp_x \psi(x)$ is bounded in its operator norm.  We
also assume that the expected Frobenius norm squared of $\psi(x)$ is
$d_\psi$, which holds, for example, if each $\psi(x)$ is unitary.

Our main theorem asserts that the expected $\ell_2$ distance is bounded below by a function of the ratio $d_\psi / \dmin$.  Roughly speaking, if we think of $\psi$ as a low-dimensional embedding of $G$, we cannot avoid a certain amount of ``distortion'' of $G$'s multiplicative structure.  We let $\opnorm{A}$ denote the operator norm, and let $\norm{A}_{\fsub}^2 = \tr A^\dagger A$ denote the Frobenius norm.  

\begin{theorem}
\label{thm:approx-irreps}
Let $G$ be a group and let $\dmin$ denote the dimension of $G$'s smallest nontrivial irrep.  For any function $\psi : G \to \U_{d_\psi}$, 
\begin{equation}
\label{eq:approx-irreps}
\Exp_{x,y \in G} \norm{\psi(xy) - \psi(x) \,\psi(y)}_{\fsub}^2 
\ge 2d_\psi \left( 1 - \bopnorm{ \Exp_x \psi(x) }^3 - \sqrt{\frac{d_\psi}{\dmin}} \right) \, .
\end{equation}
\end{theorem}

\noindent
If $A$ and $B$ are random unitary matrices of dimension $d_\psi$ distributed according to Haar measure, then $\Exp_{A,B} \norm{A-B}_{\fsub}^2 = 2d$.  Thus Theorem~\ref{thm:approx-irreps} shows that when $d_\psi / \dmin$ and $\bopnorm{ \Exp_x \psi(x) }$ are small, $\psi$ is little better than a random function from $G$ to $\U_d$ as far as acting like a representation is concerned.  

We comment that Theorem~\ref{thm:approx-irreps} holds in the more
general setting where $\psi$ is a function from $G$ to the group
$\GL_{d_\psi}$ of invertible $d_\psi$-dimensional matrices, as long as
$\psi$ is ``unitary in expectation'' in the sense that
\begin{equation}
\label{eq:cond}
\Exp_x \psi(x)^\dagger \psi(x) = \one \, .
\end{equation}

In the regime where $\psi$ is very close to a representation, our work is related to Babai, Friedl, and Luk\'acs~\cite{babai-near-focs,babai-near}.  They showed that if $\norm{\psi(xy) - \psi(x) \,\psi(y)}_{\fsub}^2$ is sufficiently small, then there is a genuine representation $\rho$ with $d_\rho=d_\psi$ such that $\rho$ is close to $\psi$.  Their definitions are slightly different; for instance, they consider uniform bounds on $\norm{\psi(xy) - \psi(x) \,\psi(y)}_{\fsub}^2$ rather than its expectation over all pairs of elements $x,y \in G$, and also place a bound on $\norm{\psi(1)-\one}_{\fsub}^2$.  Nevertheless, the Fourier-analytic proof of Theorem~\ref{thm:approx-irreps} uses similar Fourier analytic techniques as in their work.

Theorem~\ref{thm:approx-irreps} yields the following corollary, bounding the probability that $\psi(xy) = \psi(x) \,\psi(y)$ for uniformly random $x, y$:
\begin{corollary}
\label{cor:prob}
Let $G$, $\dmin$, and $\psi: G \to U_{d_\psi}$ be as in Theorem~\ref{thm:approx-irreps}.  If $x, y \in G$ are uniformly random, then
\[
\Pr[ \psi(xy) = \psi(x) \,\psi(y) ]
\le \frac{1}{2} \left( 1 + \bopnorm{ \Exp_x \psi(x) }^3 + \sqrt{\frac{d_\psi}{\dmin}} \right) \, .
\]
\end{corollary}

\noindent
When $d_\psi / \dmin$ is small, this is tight for a random function $\psi$ that sends half the elements of $G$ to $\one$ and the other half to $-\one$, where each half is chosen uniformly at random from all subsets of size $|G|/2$.

As an application of these results, we consider approximate homomorphisms $f:G \to H$ where $H$ is another finite group, bounding the probability that $f(xy) = f(x) \,f(y)$ for uniformly random pairs $x,y \in G$.  To avoid the trivial homomorphism $f(x)=1$, we require that $f$'s image is close to uniform.  For each $y \in H$ define the probability 
\[
p_f(y) = \Pr_x [f(x)=y] 
\]
that a uniformly random $x \in G$ has image $y$.  Then we bound the $\ell_2$ distance between $p_f$ and the uniform distribution $u(y)=1/|H|$, requiring that 
\begin{equation}
\label{eq:f-uniform}
\norm{p_f-u}_2^2  = \sum_{y \in H} \abs{p_f(y)-\frac{1}{|H|}}^2 \le \frac{\eps}{|H|} \, .
\end{equation}
For instance, this holds with $\eps=1$ if $p_f(y)$ is uniform on a
subgroup of $H$ of index $2$. 

We will use the fact that if $f$ is an approximate homomorphism then, for each irrep $\sigma$ of $H$, the composition $\sigma \circ f$ is an approximate representation of $G$.  Our first bound focuses on one $\sigma$ at a time.

\begin{theorem}
\label{thm:approx-hom}
Let $G$ and $H$ be finite groups, and let $\dmin$ 
denote the dimension of $G$'s smallest nontrivial irrep.  Let $f: G \to H$ such that~\eqref{eq:f-uniform} holds.  Then 
\[
\Pr[ f(xy) = f(x) \,f(y) ]
\le \frac{1}{2} \min_{\sigma \ne 1} \left( 1 + \sqrt{ \frac{\eps}{d_\sigma} } + \sqrt{\frac{d_\sigma}{\dmin}} \right) \, ,
\]
where $\sigma$ ranges over all of $H$'s nontrivial irreps.
\end{theorem}

\noindent
If $p_f$ is perfectly uniform so that $\eps=0$, this expression is minimized by $H$'s smallest nontrivial irrep $\sigma$.  In that case, $f$ cannot act very homomorphically if $H$ is much less quasirandom than $G$ is.  

Our second bound considers all of $H$'s irreps, not just the smallest one.  
Recall that the \emph{Plancherel measure} assigns each irrep $\sigma \in \wH$ the probability $P(\sigma) = d_\sigma^2/|H|$.  If $R_H$ denotes the expectation of 
$\sqrt{d_\sigma/\dmin}$ or $1$, whichever is smaller, 
\begin{equation}
\label{eq:rh}
R_H(\dmin)
= \sum_{\sigma \in \widehat{H}} \frac{d_\sigma^2}{|H|} 
\,\min\!\left( 
\sqrt{ \frac{d_\sigma}{\dmin} } \, , \; 1 \right) \, , 
\end{equation}
then we have the following.

\begin{theorem}
\label{thm:approx-hom-rh}
Let $G$ and $H$ be finite groups, and let $\dmin$ denote the dimension of $G$'s smallest nontrivial irrep.  Let $f: G \to H$ such that~\eqref{eq:f-uniform} holds.  Then
\[
\Pr[ f(xy) = f(x) \,f(y) ]
\le \norm{p_f}_2^2 + R_H(\dmin) = \frac{1 + \epsilon}{|H|} + R_H(\dmin)\, .
\]
\end{theorem}

\noindent
If $R_H$ and $\eps$ are small, i.e., if most of $H$'s irreps are much smaller than $\dmin$ and $f$'s image is close to uniform, Theorem~\ref{thm:approx-hom-rh} shows that $f$ cannot act like a homomorphism much more often than a random function from $G$ to $H$.

Finally, we give two results indicating that the bounds of Theorem~\ref{thm:approx-irreps} are essentially tight.  First we show that they are achieved exactly if $\psi$ is proportional to a minor of a genuine irreducible representation.  While these minors are not unitary, we can scale them so that they are unitary in expectation in the sense of~\eqref{eq:cond}.

\begin{theorem}
\label{thm:minor-tight}
Let $\rho: G \rightarrow \U(V)$ be an irreducible representation of $G$ of dimension $d_\rho$ and let $\Pi: V \rightarrow V$ be a projection operator onto a $d_\psi$-dimensional subspace $W \subseteq V$.  If we define
\[
\psi(x) = \sqrt{\frac{d_\rho}{d_\psi}} \,\Pi \rho(x) \Pi \, ,
\] 
then $\Exp_x \psi(x)=0$, $\Exp_x \psi(x)^\dagger \psi(x) = \Pi$, and 
\[
\Exp_{x,y \in G} \norm{\psi(xy) - \psi(x) \,\psi(y)}_{\fsub}^2 
= 2d_\psi \left( 1 - \sqrt{\frac{d_\psi}{d_\rho}} \right) \, .
\]
\end{theorem}
\noindent
Note that this precisely matches our upper bound in Theorem~\ref{thm:approx-irreps} in the case $\Exp_x \psi(x) = 0$.

In our last result, we use the polar decomposition to make these
approximate representations unitary.  This comes at some cost to the
expected Frobenius norm, but there is still a regime for $d_\psi /
d_\rho$
where we can achieve significantly stronger results than those of a
random function.  First recall that if $A$ is a $d$-dimensional
complex matrix of full rank, its \emph{polar decomposition} expresses
$A$ as the product of a unitary matrix $\tA$ and a positive
semidefinite matrix
\[
A = \tA P \, ,
\]
where
\[
\tA = A (A^\dagger A)^{-1/2} 
\quad \text{and} \quad
P = (A^\dagger A)^{1/2} \, .
\]
That is, $\tA = AP^{-1}$ where $P$ is the unique positive semidefinite matrix such that $P^2 = A^\dagger A$.  It is a simple exercise to show that $\tA$ is unitary.  More importantly, $\tA$ is the unitary matrix which is closest to $A$ in $\ell_2$ distance~\cite{fan-hoffman}.  

Then we have the following theorem.  Note that unlike Theorem~\ref{thm:minor-tight}, we now assume that $\Pi$ is chosen uniformly.  Specifically, given a fixed projection operator $\Pi'$ of rank $d_\psi$, we set $\Pi = U^\dagger \Pi' U$ where $U \in \U(V)$ is uniform according to the Haar measure.

\begin{theorem}
\label{thm:polar}
Let $\rho: G \rightarrow \U(V)$ be an irreducible representation of $G$ of dimension $d_\rho$ and let $\Pi: V \rightarrow V$ be a projection operator onto a subspace $W \subseteq V$ chosen uniformly from all projection operators of rank $d_\psi$.  Let $\psi(x)$ be defined as in Theorem~\ref{thm:minor-tight}, and let $\tpsi(x)$ be the unitary part of its polar decomposition.  Then
\[
\Exp_{\Pi} \Exp_{x,y \in G} \norm{\tpsi(xy) - \tpsi(x) \,\tpsi(y)}_{\fsub}^2 
\le 2 d_\psi \left( 4 \left( 1 - \sqrt{\frac{d_\psi}{d_\rho}} \right) + 6 \left( 1-\frac{d_\psi}{d_\rho} \right) \right) 
\, .
\]
It follows that there exists a particular projection operator $\Pi$ satisfying
the bound above.
\end{theorem}
The difference between Theorems~\ref{thm:minor-tight} and~\ref{thm:polar} is the cost of making $\psi(x)$ unitary---it comes from bounding the expected $\ell_2$ distance between $\psi(x)$ and $\tpsi(x)$ and using the triangle inequality.  While it is intuitive that this cost is nonzero, we have not attempted to optimize this bound.  Nevertheless, even this relatively crude bound shows that there exist unitary approximate representations that perform noticeably better than random matrices---that is, for which 
\[
\frac{1}{2 d_\psi} \,\Exp_{x,y \in G} \norm{\tpsi(xy) - \tpsi(x) \,\tpsi(y)}_{\fsub}^2 \le \alpha 
\]
for some $\alpha < 1$,  whenever 
\[
d_\psi / d_\rho > \frac{31-2\sqrt{58}}{18} = 0.876\ldots 
\]
We conjecture that approximate representations exist with $\alpha < 1$ whenever $d_\psi / d_\rho > 0$.

Proofs are given in the following three sections.

\section{Bounds on approximate representations}

In this section we prove Theorem~\ref{thm:approx-irreps} and Corollary~\ref{cor:prob}.  In the process, 
we set our conventions for the nonabelian Fourier transform, and prove several inequalities that we will apply later on.  

\begin{proof}[{\bf Proof of Theorem~\ref{thm:approx-irreps}}]
For any $A, B$ we have
\[
\norm{A-B}_{\fsub}^2 
= \tr (A-B)^\dagger (A-B) 
= \norm{A}_{\fsub}^2 + \norm{B}_{\fsub}^2 - 2 \Re \tr A^\dagger B \, . 
\]
Since $z=xy$ is uniformly random whenever $x$ and $y$ are, 
\begin{equation}
\label{eq:alternately}
\Exp_{x,y} \norm{\psi(xy) - \psi(x) \,\psi(y)}_{\fsub}^2 
= \Exp_{z} \frob{\psi(z)}^2 + \Exp_{x,y} \frob{\psi(x) \,\psi(y)}^2 - 2 \Re \Exp_{x,y} \tr \psi(xy)^\dagger \,\psi(x) \,\psi(y) \, . 
\end{equation}
If $\psi$ is unitary in expectation, \eqref{eq:cond} implies
\begin{equation}
\label{eq:frob-z}
\Exp_{z} \frob{\psi(z)}^2 = \tr \Exp_z \psi(z)^\dagger \psi(z) = \tr \one = d_\psi \, ,
\end{equation}
and of course this holds identically if $\psi$ is unitary.  Similarly, 
\begin{align*}
\Exp_{x,y} \frob{\psi(x) \,\psi(y)}^2
&= \tr \Exp_{x,y} \psi(y)^\dagger \psi(x)^\dagger \psi(x) \psi(y) \\
&= \tr \left( \Exp_x \psi(x)^\dagger \psi(x) \right) \left( \Exp_y \psi(y) \psi(y)^\dagger \right) 
= \tr \one = d_\psi \, .
\end{align*}
Then~\eqref{eq:alternately} becomes
\begin{equation}
\label{eq:dist-trace}
\Exp_{x,y} \norm{\psi(xy) - \psi(x) \,\psi(y)}_{\fsub}^2 
= 2d_\psi \left(1 - \frac{1}{d_\psi} \Re \Exp_{x,y} \tr \psi^\dagger(xy) \,\psi(x) \,\psi(y) \right) \, .
\end{equation}
Thus we will focus on estimating the expected trace
\begin{equation}
\label{eq:exp}
\Exp_{x,y} \tr \psi^\dagger(xy) \,\psi(x) \,\psi(y) \, .
\end{equation}
Note that if $\psi$ is a genuine representation, $\psi^\dagger(xy) \,\psi(x) \,\psi(y) = \one$ and this trace is identically $d_\psi$.

We rely on nonabelian Fourier analysis, for which we refer the reader to~\cite{Serre77}.  In order to establish our notation and choice of normalizations, let $f: G \rightarrow \C$ and let $\rho: G \rightarrow \U_d$ be an irreducible unitary representation of $G$ or ``irrep'' for short, and let $\wG$ denote the set of irreps of $G$.  We adopt the Fourier transform
\[
\wf(\rho) = \frac{1}{|G|} \sum_{x \in G} f(x) \,\rho^\dagger(x)
= \Exp_x f(x) \,\rho^\dagger(x) \, ,
\]
in which case we have the Fourier inversion formula
\[
f(x) = \sum_{\rho \in \wG} d_\rho \tr \left( \wf(\rho) \,\rho(x) \right) \, .
\]
The Fourier transform preserves inner products in the sense that
\begin{equation}
\label{eq:inner}
\langle f, g \rangle = \sum_x f(x)^* g(x) 
= |G| \sum_\rho d_\rho \tr\left( \wf(\rho)^\dagger \,\wg(\rho) \right) \, .
\end{equation}
In particular, we have Plancherel's identity,
\begin{equation}
\label{eq:planch}
\norm{f}_2^2 = \sum_x \abs{f(x)}^2 
= |G| \sum_\rho d_\rho \bigl\|\wf(\rho)\bigr\|_{\fsub}^2 \, .
\end{equation}

Since $\psi$ is a matrix-valued function, each entry $\psi(x)^i_j$ has
its own Fourier transform.  Therefore, we may treat the Fourier transform $\wpsi$ as a tensor with four indices,
\[
\wpsi(\rho)^{ik}_{j\ell} = \Exp_x \psi(x)^i_j \,\rho^\dagger(x)^k_\ell 
\quad \text{or} \quad 
\wpsi(\rho) = \Exp_x \left[ \psi(x) \otimes \rho^\dagger(x) \right] \, . 
\]
The Fourier inversion formula can then be expressed as a partial trace. We adopt the Einstein summation convention, where any index appearing twice is automatically summed over.  For instance, $(AB)^i_j = A^i_k B^k_j$ and $\tr A = A^i_i$.  Then
\[
\psi(x)^i_j = \sum_\rho d_\rho \,\wpsi(\rho)^{ik}_{j\ell} \,\rho(x)^\ell_k \, .
\]
Plancherel's identity becomes
\begin{equation}
\label{eq:psi-planch}
\Exp_x \norm{\psi(x)}_{\fsub}^2 
= \sum_\rho d_\rho \norm{\wpsi(\rho)}_{\fsub}^2
= \sum_\rho d_\rho \sum_{i,j,k,l} \abs{\wpsi(\rho)^{ik}_{j\ell}}^2 \, .
\end{equation}


We compute the trace~\eqref{eq:exp} by evaluating it in the Fourier basis.  First write
\begin{align}
\tr \psi^\dagger(xy) \,\psi(x) \,\psi(y) 
&= \psi^\dagger(xy)^i_j \,\psi(x)^j_k \,\psi(y)^k_i \nonumber \\
&= \sum_{\rho, \sigma, \tau \in \wG} d_\rho \,d_\sigma \,d_\tau 
\,\wpsidag(\rho)^{ia}_{jb} \,\wpsi(\sigma)^{jd}_{ke} \,\wpsi(\tau)^{kf}_{ig} 
\,\rho(xy)^b_a \,\sigma(x)^e_d \,\tau(y)^g_k \nonumber \\
&= \sum_{\rho, \sigma, \tau \in \wG} d_\rho \,d_\sigma \,d_\tau 
\,\wpsidag(\rho)^{ia}_{jb} \,\wpsi(\sigma)^{jd}_{ke} \,\wpsi(\tau)^{kf}_{ig} 
\,\rho(x)^b_c \,\rho(y)^c_a \,\sigma(x)^e_d \,\tau(y)^g_f \, .
\label{eq:tr1}
\end{align}
Schur's lemma implies
\begin{equation}
\label{eq:cupcap0}
\Exp_x \left[ \rho(x)^b_c \,\sigma(x)^e_d \right]
= \frac{1}{d_\rho} 
\begin{cases} 
\delta^{be} \,\delta_{cd} & \sigma=\rho^* \\
0 & \sigma \ne \rho^* \, .
\end{cases}
\end{equation}
Thus taking the expectation over $x$ and $y$ turns~\eqref{eq:tr1} into
\begin{align*}
\Exp_{x,y} \tr \psi^\dagger(xy) \,\psi(x) \,\psi(y)
&= \sum_\rho d_\rho 
\,\wpsidag(\rho)^{ia}_{jb} \,\wpsi(\rho^*)^{jd}_{ke} \,\wpsi(\rho^*)^{kf}_{ig} 
\,\delta^{be} \,\delta_{cd} \,\delta^{cg} \,\delta_{af} \\
&= \sum_\rho d_\rho 
\,\wpsidag(\rho)^{ia}_{jb} \,\wpsi(\rho^*)^{jd}_{kb} \,\wpsi(\rho^*)^{ka}_{id} \, . 
\end{align*}
We rearrange this slightly, writing $\wpsi(\rho^\dagger)^{ik}_{j\ell}$ for the partial transpose $\wpsi(\rho^*)^{i\ell}_{jk}$.  Then 
\begin{equation}
\label{eq:wpsi}
\wpsi(\rho^\dagger) = \Exp_x \left[ \psi(x) \otimes \rho(x) \right] \, ,
\end{equation}
and
\begin{equation}
\label{eq:tr2}
\Exp_{x,y} \tr \psi^\dagger(xy) \,\psi(x) \,\psi(y)
= \sum_\rho d_\rho 
\,\wpsidag(\rho)^{ia}_{jb} \,\wpsi(\rho^\dagger)^{jb}_{kd} \,\wpsi(\rho^\dagger)^{kd}_{ia} \, . 
\end{equation}
If we view $\wpsi(\rho^\dagger)$ as a linear operator on $\C^d \otimes \C^{d_\rho}$, then 
\[
\wpsidag(\rho) 
= \left( \wpsi(\rho^\dagger) \right)^\dagger 
= \Exp_x \left[ \psi^\dagger(x) \otimes \rho^\dagger(x) \right] \, ,
\] 
and we can write
\begin{equation}
\label{eq:tr3}
\Exp_{x,y} \tr \psi^\dagger(xy) \,\psi(x) \,\psi(y)
= \sum_\rho d_\rho \tr \,\wpsi(\rho^\dagger) \,\left( \wpsi(\rho^\dagger) \right)^\dagger \,\wpsi(\rho^\dagger) \, .
\end{equation}

We remark that in the case where $\psi$ is an irrep the only irrep contributing to the sum~\eqref{eq:tr3} is $\psi^*$ since, as in~\eqref{eq:cupcap0}, we have
\[
\wpsi(\rho^\dagger)^{ik}_{j\ell} 
= \Exp_x \left[ \psi(x)^i_j \,\rho(x)^k_\ell \right] 
= \frac{1}{d_\psi} 
\begin{cases}
\delta^{ik} \,\delta_{j\ell} & \rho=\psi^* \, , \\
0 & \rho \ne \psi^* \, .
\end{cases}
\]

Let $\Pi$ denote the operator $(1/d_\psi) \,\delta^{ik}
\,\delta_{j\ell}$.  Diagrammatically, $\Pi$ is proportional to the
``cupcap.''  It is a one-dimensional projection operator, equal to the
outer product of the
vector 
\[
(1/\sqrt{d_\psi}) \sum_i e_i \otimes e_i
\]
with itself, where $e_i$ denotes the $i$th basis vector.  Since $\Pi$
is Hermitian, \eqref{eq:tr3} implies that
\[
\Exp_{x,y} \tr \psi^\dagger(xy) \,\psi(x) \,\psi(y) = d_\psi \tr \Pi^3 = d_\psi \tr \Pi = d_\psi
\]
which holds since $\psi^\dagger(xy) \,\psi(x) \,\psi(y) = \one$. 

Returning to~\eqref{eq:tr3}, since $A^\dagger A$ is positive for any $A$ we have
\[
\tr A A^\dagger A \le \opnorm{A} \norm{A}_{\fsub}^2 \, .
\]
This gives
\begin{equation}
\label{eq:tr-trivial}
\Exp_{x,y} \tr \psi^\dagger(xy) \,\psi(x) \,\psi(y)
\le \sum_\rho d_\rho \,\bopnorm{\wpsi(\rho^\dagger)} \,\bnorm{\wpsi(\rho^\dagger)}_{\fsub}^2 \, . 
\end{equation}
We separate out the term corresponding to the trivial representation $\rho=1$, for which $\wpsi(1) = \Exp_x \psi(x)$.  Since $\frob{A}^2 \le d \opnorm{A}^2$ for any $d$-dimensional matrix $A$, we have
\begin{align}
\Exp_{x,y} \tr \psi^\dagger(xy) \,\psi(x) \,\psi(y) 
& \le \bopnorm{ \Exp_x \psi(x) } \,\bnorm{ \Exp_x \psi(x) )}_{\fsub}^2
+ \sum_{\rho \ne 1} d_\rho \,\bopnorm{\wpsi(\rho^\dagger)} \,\bnorm{\wpsi(\rho^\dagger)}_{\fsub}^2 \label{eq:tr-bound-unsimp} \\
& \le d_\psi \bopnorm{ \Exp_x \psi(x) }^3 
+ \left( \max_{\rho \ne 1} \bopnorm{\wpsi(\rho^\dagger)} \right) \sum_\rho d_\rho \,\bnorm{\wpsi(\rho^\dagger)}_{\fsub}^2 \nonumber \\
& = d_\psi \bopnorm{ \Exp_x \psi(x) }^3 
+ \left( \max_{\rho \ne 1} \bopnorm{\wpsi(\rho^\dagger)} \right) \Exp_x \bnorm{\psi(x)}_{\fsub}^2 \nonumber \\
& = d_\psi \left( \bopnorm{ \Exp_x \psi(x) }^3
+ \max_{\rho \ne 1} \bopnorm{\wpsi(\rho^\dagger)} \right) \, , 
\label{eq:tr-bound}
\end{align}
where we used Plancherel's identity~\eqref{eq:psi-planch} in the third line and the fact that $\Exp_x \bnorm{\psi(x)}_{\fsub}^2 = d_\psi$ is unitary, or unitary in expectation, in the fourth.  This is analogous to the Fourier-analytic treatment of the Blum-Luby-Rubinfeld linearity test~\cite{blum-luby-rubinfeld,bellare-et-al}.  

Our next goal is to bound the operator norm of $\wpsi(\rho^\dagger)$.
Let $V$ and $W$ denote the spaces on which $\psi$ and $\rho$ act,
respectively.  Then $\bopnorm{ \wpsi(\rho^\dagger) }$ is the maximum, taken over all vectors $u \in V \otimes W$ of norm $1$, of $\inner{u}{\wpsi(\rho^\dagger) \,u}$.  Using the Schmidt decomposition we can write
\[
u = \sum_i \alpha_i v_i \otimes w_i
\]
where $\{v_i\}$ and $\{w_i\}$ are orthogonal bases for $V$ and a $d_\psi$-dimensional subspace of $W$ respectively, and where $\sum_i \abs{\alpha_i}^2 = 1$.  Then separating the tensor product and using Cauchy-Schwarz gives
\begin{align*}
\inner{u}{\wpsi(\rho^\dagger) \,u}
&= \sum_{i,j} \alpha_i^* \alpha_j \inner{v_i \otimes w_i}{\wpsi(\rho^\dagger) \,v_j \otimes w_j} \\
&= \sum_{i,j} \alpha_i^* \alpha_j \inner{v_i \otimes w_i}{ \left( \Exp_x \left[ \psi(x) \otimes \rho(x) \right] \right) \,v_j \otimes w_j} \\
&= \sum_{i,j} \alpha_i^* \alpha_j \Exp_x \left[ \inner{v_i}{\psi(x) \,v_j} \inner{w_i}{\rho(x) \,w_j} \right] \\
&\le \sum_{i,j} \alpha_i^* \alpha_j \sqrt{ \left( \Exp_x \inner{v_i}{\psi(x) \,v_j}^2 \right) \left( \Exp_x \inner{w_j}{\rho(x) \,w_j}^2 \right) } 
\, .
\end{align*}
By Schur's lemma we have $\Exp_x \inner{w_i}{\rho(x) \,w_j}^2 = 1/d_\rho$, giving
\[
\inner{u}{\wpsi(\rho^\dagger) \,u} 
\le \frac{1}{\sqrt{d_\rho}} \sum_{i,j} \alpha_i^* \alpha_j \sqrt{ \Exp_x \inner{v_i}{\psi(x) \,v_j}^2  } \, . 
\]
Another application of Cauchy-Schwarz and the fact that $\norm{\psi(x)}_{\fsub}^2 = d_\psi$ gives 
\begin{align}
\inner{u}{\wpsi(\rho^\dagger) \,u} 
&\le \frac{1}{\sqrt{d_\rho}} \sqrt{ \Bigl( \sum_{i,j} \bigl|\alpha_i\bigr|^2 \bigl|\alpha_j\bigr|^2 \Bigr) \sum_{ij} \Exp_x \inner{v_i}{\psi(x) \,v_j}^2 } 
\nonumber \\
&= \frac{1}{\sqrt{d_\rho}} \sqrt{ \Bigl( \sum_i \abs{\alpha_i}^2 \Bigr)^2 \Exp_x \Bigl[ \sum_{i,j} \inner{v_i}{\psi(x) \,v_j}^2 \Bigr] } 
\nonumber \\
&= \frac{1}{\sqrt{d_\rho}} \sqrt{ \Exp_x \norm{\psi(x)}_{\fsub}^2 } 
\nonumber \\
&= \sqrt{ \frac{d_\psi}{d_\rho} } \, ,
\label{eq:dpsi-drho}
\end{align} 
where we again used the fact that $\psi(x)$ is unitarity, or unitary in expectation, in the fourth line.

Combining~\eqref{eq:tr-bound} with~\eqref{eq:dpsi-drho} and using our hypothesis that $\min_{\rho \ne 1} d_\rho = \dmin$ then gives
\begin{equation}
\label{eq:tr-dpsi-dmin}
\Exp_{x,y} \tr \psi^\dagger(xy) \,\psi(x) \,\psi(y) 
\le d_\psi \left( \bopnorm{ \Exp_x \psi(x) }^3 + \sqrt{ \frac{d_\psi}{\dmin} } \right) \; .
\end{equation}
Finally, combining this with~\eqref{eq:dist-trace} completes the proof.
\end{proof}

\begin{proof}[{\bf Proof of Corollary~\ref{cor:prob}}]
For any $A, B \in \U_d$ with $A \ne B$ we have $\frob{A-B}^2 \le 4d$.  Thus
\[
\Exp_{x,y} \norm{\psi(xy) - \psi(x) \,\psi(y)}_{\fsub}^2 \le 4d_\psi \Pr[ \psi(xy) \ne \psi(x) \,\psi(y) ] \, ,
\]
and so
\[
\Pr[ \psi(xy) = \psi(x) \,\psi(y) ] 
\le 1 - \frac{1}{4d_\psi} \,\Exp_{x,y} \norm{\psi(xy) - \psi(x) \,\psi(y)}_{\fsub}^2 \, .
\]
Combining this with the bound~\eqref{eq:approx-irreps} completes the proof.
\end{proof}

\section{Approximate homomorphisms}

In this section we prove Theorems~\ref{thm:approx-hom} and~\ref{thm:approx-hom-rh}, bounding the extent to which a function from one finite group to another can act like a homomorphism.

\begin{proof}[{\bf Proof of Theorem~\ref{thm:approx-hom}}]
Let $\sigma$ be an irreducible representation of $H$.  We treat $\psi_\sigma = \sigma \circ f$ as an approximate representation of $G$ of dimension $d_\sigma$.  To bound $\bopnorm{ \Exp_x \psi(x) }$, note that 
\[
\Exp_x \psi_\sigma(x) 
= \sum_{y \in H} p_f(y) \,\sigma(y) 
= \sum_{y \in H} (p_f-u)(y) \,\sigma(y) 
= |H| \,\widehat{(p_f-u)}(\sigma) \, .
\]
where we used the fact that $\Exp_y \sigma(y) = 0$.  
Then we have
\[
\bopnorm{ \Exp_x \psi_\sigma(x) }^2
\le \bnorm{ \Exp_x \psi_\sigma(x) }_{\fsub}^2 
\le |H|^2 \bnorm{ \widehat{(p_f-u)}(\sigma) }_{\fsub}^2 
\le \frac{|H|}{d_\sigma} \norm{p_f-u}_2^2
\le \frac{\eps}{d_\sigma} \, ,
\]
where we used Plancherel's identity~\eqref{eq:planch} in the third inequality.  Thus
\begin{equation}
\label{eq:psi-sigma}
\bopnorm{ \Exp_x \psi_\sigma(x) } \le \sqrt{ \frac{\eps}{d_\sigma} } \, , 
\end{equation}
and applying Corollary~\ref{cor:prob} completes the proof.
\end{proof}

\begin{proof}[{\bf Proof of Theorem~\ref{thm:approx-hom-rh}}]
  We let $R = \sum_{\sigma \in \hat{H}} d_\sigma \chi_\sigma$ denote
  the regular representation of $H$. As above, for an irreducible
  representation $\sigma$ of $H$ we define $\psi_\sigma = \sigma \circ
  f$. Observe that
\begin{align*}
 \Pr_{x, y}[ f(xy) = f(x) \,f(y) ] &= \frac{1}{|H|} \Exp_{x,y} [
 R(f(xy)^{-1}  f(x) f(y))] = \sum_{\sigma \in \widehat{H}}\frac{d_\sigma}{|H|} 
 \Exp_{x,y} \bigl[\chi_\sigma \bigl( f(xy)^{-1} \,f(x) \,f(y) \bigr)\bigr] \\
&= \sum_{\sigma \in \widehat{H}} \frac{d_\sigma}{|H|} 
 \Exp_{x,y} \bigl[\tr \psi_\sigma^\dagger(xy) \,\psi_\sigma(x) \,\psi_\sigma(y) \bigr)\bigr] \\
 &\le \frac{1}{|H|} + \underbrace{\sum_{\sigma \ne 1} \frac{d_\sigma}{|H|} 
 \bopnorm{\Exp_x \psi_\sigma(x) }\bnorm{\Exp_x \psi_\sigma(x)
 }_{\fsub}^2}_{(*)} + \sum_{\sigma \ne 1}
 \frac{d_\sigma^2}{|H|} \min\left( \sqrt{ \frac{d_\sigma}{\dmin} } \,
   , \; 1 \right) \, ,
\end{align*}
the last inequality following from equation~\eqref{eq:tr-bound-unsimp}. Focusing on the term $(*)$ above, 
\begin{align*}
\sum_{\sigma \ne 1} \frac{d_\sigma}{|H|} 
 \bopnorm{\Exp_x \psi_\sigma(x) }\bnorm{\Exp_x \psi_\sigma(x) }_{\fsub}^2
& \leq \sum_{\sigma \ne 1} \frac{d_\sigma}{|H|} \bnorm{\Exp_x \psi_\sigma(x) }_{\fsub}^2\\
& = \frac{\bnorm{ \Exp_x R(f(x)) }_{\fsub}^2 - 1}{|H|}\\
&  = \frac{\tr
  \left(\Exp_x \Exp_y R(f (x)^{-1})R(f(y))\right)}{|H|} -
\frac{1}{|H|}\\
& = \Pr_{x,y} [ f(x) = f(y) ] - \frac{1}{|H|} \, .
\end{align*}
Hence
\begin{align*}
 \Pr_{x, y}[ f(xy) = f(x) \,f(y) ] 
 & \leq \Pr_{x,y} [ f(x) = f(y) ] 
 + \sum_{\sigma \ne 1} \frac{d_\sigma^2}{|H|} \min\left( \sqrt{ \frac{d_\sigma}{\dmin} } \, , \; 1 \right) \\
 & \leq \Pr_{x,y} [ f(x) = f(y) ] + R_H \, .
\end{align*}

Finally, since
\[
\Pr[ f(x) = f(y) ] = \sum_{x \in H} p_f(x)^2 = \norm{p_f}_2^2 = \norm{u}_2^2 + \norm{ p_f - u}_2^2 \le \frac{1 + \epsilon}{|H|} \, , 
\]
the statement of the theorem follows.
\end{proof}

\section{Minors of representations and their polar decompositions}

In this section we prove Theorems~\ref{thm:minor-tight} and~\ref{thm:polar}. 

\begin{proof}[Proof of Theorem~\ref{thm:minor-tight}] 
Let $\rho: G \rightarrow \U(V)$ be an irreducible representation of $G$ of dimension $d_\rho$ and let $\Pi: V \rightarrow V$ be a projection operator of rank $d_\psi$. Treating the image of $\Pi$ as a subspace $W$, we consider the function $\psi: G \rightarrow \End(W)$ given by
\[
\psi(x) = \sqrt{\frac{d_\rho}{d_\psi}} \,\Pi \rho(x) \Pi \, .
\]
Then 
\[
\Exp_x \psi(x) = \sqrt{\frac{d_\psi}{d_\rho}} \,\Pi \left( \Exp_x \rho(x) \right) \Pi = 0 \, .
\]
Moreover, since $\rho$ is irreducible, Schur's lemma gives
\begin{equation}
\label{eq:symmetrized-proj}
\Exp_x \rho(x)^\dagger \Pi \rho(x) 
= \frac{\rk \Pi}{d_\rho} \one 
= \frac{d_\psi}{d_\rho} \one \, .
\end{equation}
Thus 
\[
\Exp_x \psi(x)^\dagger \psi(x) 
= \frac{d_\rho}{d_\psi} \Exp_x \Pi \rho(x)^\dagger \Pi \rho(x) \Pi
= \frac{d_\rho}{d_\psi} \Pi \left( \Exp_x \rho(x)^\dagger \Pi \rho(x) \right) \Pi
= \Pi \, . 
\]
Since $\Pi$ is the identity on the subspace $W$, $\psi(x)$ is unitary in expectation.


As in the proof of Theorem~\ref{thm:approx-irreps}, we then have
\begin{equation}
\label{eq:minor-frob}
\Exp_{x,y} \norm{\psi(xy) - \psi(x) \,\psi(y)}_{\fsub}^2 
= 2d_\psi \left(1 - \frac{1}{d_\psi} \Re \Exp_{x,y} \tr \psi^\dagger(xy) \,\psi(x) \,\psi(y) \right) \, .
\end{equation}
Since $\rho$ is a genuine representation, $\rho(xy) = \rho(x) \,\rho(y)$ and
\begin{align*}
\Exp_{x,y} \tr \psi^\dagger(xy) \,\psi(x) \,\psi(y) 
&= \left( \frac{d_\rho}{d_\psi} \right)^{\!3/2} \Exp_{x,y} \tr \Pi \rho^\dagger(xy) \Pi \rho(x) \Pi \rho(y) \Pi \\
&= \left( \frac{d_\rho}{d_\psi} \right)^{\!3/2} \Exp_{x,y} \tr \Pi \rho(y)^\dagger \rho(x)^\dagger \Pi \rho(x) \Pi \rho(y) \Pi \\
&= \left( \frac{d_\rho}{d_\psi} \right)^{\!3/2} 
\tr \left[ \left( \Exp_y \rho(y) \Pi \rho(y)^\dagger \right) \Exp_x \left( \rho(x)^\dagger \Pi \rho(x) \right) \Pi \right] \\
&= \sqrt{\frac{d_\psi}{d_\rho}} \, \tr \Pi 
= d_\psi \sqrt{\frac{d_\psi}{d_\rho}} \, .
\end{align*}
Combining this with~\eqref{eq:minor-frob} completes the proof.
\end{proof}

%
%

\begin{proof}[Proof of Theorem~\ref{thm:polar}]
The squared $\ell_2$ distance between a matrix $A$ and the unitary part of its polar decomposition, $\tA = A(A^\dagger A)^{-1/2}$, is 
\[
\frob{A-\tA}^2 
= \frob{A-A(A^\dagger A)^{-1/2}}^2 
= \tr A^\dagger A - 2 \tr (A^\dagger A)^{1/2} + d
= \sum_\lambda (\lambda-1)^2 \, .
\]
Here $\lambda$ ranges over the singular values of $A$, i.e., the square roots of the eigenvalues of $A^\dagger A$.  For any $\lambda \ge 0$ we have
\begin{equation}
\label{eq:lambda}
(\lambda-1)^2 \le (\lambda-1)^2 (\lambda+1)^2 = (\lambda^2-1)^2 \, .
\end{equation}
Thus the distance between $A$ and $\tA$ is at most the distance between $A^\dagger A$ and the identity, 
\begin{equation}
\label{eq:ata}
\frob{A-\tA}^2 \le \sum_\lambda (\lambda^2-1)^2 = \frob{A^\dagger A - \one}^2 \, . 
\end{equation}

Let $\psi(x) = \sqrt{d_\rho/d_\psi} \,\Pi \rho(x) \Pi$.  Since $\Pi$ is the identity on the subspace $W$, the rest of our proof consists of bounding
\begin{align}
\Exp_{\Pi,x} \frob{\psi(x) - \tpsi(x)}^2 
&\le \Exp_{\Pi,x} \frob{\psi(x)^\dagger \psi(x)-\Pi}^2 \nonumber \\
&= \Exp_{\Pi,x} \frob{\psi(x)^\dagger \psi(x)}^2 - 2 \tr \Exp_x \psi(x)^\dagger \psi(x) + d_\psi \nonumber \\
&= \Exp_{\Pi,x} \frob{\psi(x)^\dagger \psi(x)}^2 - d_\psi \, , \label{eq:firstwenotethat}
\end{align}
where in the last line we used the fact, proved in Theorem~\ref{thm:minor-tight}, that $\psi$ is unitary in expectation.  We will then use the triangle inequality to bound 
\[
\Exp_{\Pi} \Exp_{x,y} \frob{\tpsi(xy) - \tpsi(x) \,\tpsi(y)}^2 \, .
\]

We write
\begin{equation}
\label{eq:pirho4}
\Exp_{\Pi,x} \frob{\psi(x)^\dagger \psi(x)}^2 
= \left( \frac{d_\rho}{d_\psi} \right)^{\!2} \tr \Exp_{\Pi,x} \left[ \Pi \rho^\dagger(x) \Pi \rho(x) \Pi \rho^\dagger(x) \Pi \rho(x) \right]\,.
\end{equation}
We can view this trace as a contraction of two tensors.  One is $\rho \otimes \rho^\dagger \otimes \rho \otimes \rho^\dagger$.  Since we can take the expectation over all $\Pi$ by conjugating a particular $\Pi$ by a random unitary $U \in \U_{d_\rho}$, the other is the ``twirl'' of $\Pi^{\otimes 4}$, namely 
\[
\Upsilon = \Exp_U (U^\dagger \Pi U)^{\otimes 4} \, . 
\]
Since $\Upsilon$ commutes with the diagonal action of $U(d_\rho)$, it is a member of the commutant, and hence an element of the group algebra $\C[S_4]$.  Thus we can write
\[
\Upsilon = \sum_{\pi \in S_4} \upsilon(\pi) \cdot \pi \, ,
\]
where we identify each $\pi \in S_4$ with its action on $V^{\otimes
  4}$.  Moreover, since $\Upsilon$ commutes with any $\pi \in S_4$ the
coefficients $\upsilon(\pi)$ form a \emph{class function}: $\upsilon$
is constant on each conjugacy class and lies in the linear span of the
characters of $S_4$.  

We can compute the coefficients $\upsilon(\pi)$ as follows.  For any permutation $\sigma \in S_4$, we have
\[
\tr T \sigma = \tr \Pi^{\otimes 4} \sigma = d_\psi^{c(\sigma)} \, ,
\]
where $c(\sigma)$ is the number of cycles in $\sigma$.  For any $\lambda \in \wS_4$, the inner product of the character $\chi_\lambda$ with the function $d^{c(\cdot)}$ is 
\[
\inner{\chi_\lambda}{d^{c(\cdot)}} = T(\lambda,d) \, ,
\]
where $T(\lambda,d)$ denotes the number of semistandard tableaux of shape $\lambda$ and content in $\{1,\ldots,d\}$.  The multiplicity of $\lambda$ in $(\C^d)^{\otimes 4}$ is also $T(\lambda,d)$, so taking traces gives
\[
\upsilon(\pi) = \sum_{\lambda \in \wS_4} d_\lambda \chi_\lambda(\pi) \frac{T(\lambda,d_\psi)}{T(\lambda,d_\rho)} \, .
\]

A somewhat lengthy calculation gives the following coefficients for each of the five conjugacy classes in $S_4$:
\begin{align*}
\upsilon(1) 
&= \left( \frac{d_\psi}{d_\rho} \right)^{\!4} 
+ O(d_\rho^{-2}) \\
\upsilon((12)) 
&= \frac{1}{d_\rho} \left( \frac{d_\psi}{d_\rho} \right)^{\!3} \left( 1 - \frac{d_\psi}{d_\rho}  \right) 
+ O(d_\rho^{-3}) \\ 
\upsilon((123)) 
&= \frac{1}{d_\rho^2} \left( \frac{d_\psi}{d_\rho} \right)^{\!2} \left( 2 \,\frac{d_\psi}{d_\rho} - 1 \right) \left( 1- \frac{d_\psi}{d_\rho} \right) 
+ O(d_\rho^{-4}) \\ 
\upsilon((12)(34)) 
&= \frac{1}{d_\rho^2} \left( \frac{d_\psi}{d_\rho} \right)^{\!2} \left( 1- \frac{d_\psi}{d_\rho} \right)^{\!2}
+ O(d_\rho^{-4}) \\ 
\upsilon((1234)) 
&= \frac{1}{d_\rho^3} \left( 5 \left( \frac{d_\psi}{d_\rho} \right)^{\!2} - 5 \,\frac{d_\psi}{d_\rho} + 1 \right) \left( 1- \frac{d_\psi}{d_\rho} \right) \frac{d_\psi}{d_\rho} 
+ O(d_\rho^{-5}) 
\end{align*}
Note that $\upsilon(\pi)$ scales as $d_\rho^{-t(\pi)}$ where $t(\pi)=4-c(\pi)$ is the transposition distance, i.e., the minimum number of transpositions whose product gives $x$.

When $\pi$ is the identity, one of the pairs of transpositions (12)(34), and four of the 3-cycles (123), we get $\rho(x) \,\rho^\dagger(x) \,\rho(x) \,\rho^\dagger(x) = \one$, contributing $d_\rho$ to the trace.  Two of the six transpositions $(12)$ contract with $\rho \otimes \rho^\dagger \otimes \rho \otimes \rho^\dagger$ to give $\rho(x) \,\rho^\dagger(x) \otimes \rho(x) \,\rho^\dagger(x) = \one \otimes \one$, which has trace $d_\rho^2$.  These are the leading terms, and we get
\begin{align*}
\Exp_{\Pi,x} \tr & \left[ \Pi \rho^\dagger(x) \Pi \rho(x) \Pi \rho^\dagger(x) \Pi \rho(x) \right] \\
&= \big( \upsilon(1) + \upsilon((12)(34)) + 4\upsilon((123)) \big) \,d_\rho + 2 \upsilon((12)) \,d_\rho^2 + O(d_\rho^{-1}) + \text{other terms} \\
&= d_\rho \left( \frac{d_\psi}{d_\rho} \right)^{\!3} \left( 2 - \frac{d_\psi}{d_\rho} \right) + O(d_\rho^{-1}) + \text{other terms} 
 \, ,
\end{align*}
where here and in the sequel $O(\cdot)$ refers to the limit $d_\rho \to \infty$ while $d_\psi / d_\rho$ stays constant.

To bound the other terms, let $m$ denote the total multiplicity of irreducible representations appearing in the decomposition of $\rho \otimes \rho^*$.  Then
\[
\Exp_{x \in G} \abs{\chi_\rho(x)}^4 = m 
\quad \text{and} \quad
\Exp_{x \in G} \abs{\chi_\rho(x^2)}^2 \le m \, . 
\]
The first of these follows from Schur's lemma, since for any irrep $\tau$ we have $\Exp_{x \in G} \abs{\chi_\tau(x)}^2 = 1$.   The second follows from the \emph{Frobenius-Schur indicator}, which for any irrep $\tau$ is
\[
\Exp_{x \in G} \chi_\tau(x^2) = \begin{cases} 
+1 & \mbox{if $\tau$ is real,} \\
0 & \mbox{if $\tau$ is complex,} \\
-1 & \mbox{if $\tau$ is quaternionic.} 
\end{cases} 
\]

The other terms include contractions such as $(\rho \rho^\dagger \rho) \otimes \rho^\dagger = \rho\otimes \rho^\dagger$, $(\rho \rho^\dagger) \otimes \rho \otimes \rho^\dagger = \one \otimes \rho \otimes \rho^\dagger$, and so on.  These terms scale as
\begin{equation}
\label{eq:other-terms}
\begin{gathered}
\big( 4 \upsilon((12)) + 4 \upsilon((1234)) \big) \,\Exp_x \abs{\chi_\rho(x)}^2 = O(d_\rho^{-1}) \\
4 \upsilon((123)) \,d_\rho \Exp_x \abs{\chi_\rho(x)}^2 = O(d_\rho^{-1}) \\
\upsilon((1234)) \,\Exp_x \abs{\chi_\rho(x)}^4 = O(d_\rho^{-3} m) \\
\upsilon((1234)) \,\Exp_x \abs{\chi_\rho(x^2)}^2 = O(d_\rho^{-3} m) \\
2 \upsilon((12)(34)) \,\Exp_x \chi_\rho(x^2)^* \chi_\rho(x)^2 
= O(d_\rho^{-2}) \sqrt{ \Exp_x \abs{\chi_\rho(x^2)}^2 \,\Exp_x \abs{\chi_\rho(x)}^4 }
= O(d_\rho^{-2} m) \, . 
\end{gathered}
\end{equation}

If $G$ is quasirandom, with $\dmin$ the dimension of its smallest nontrivial irrep, then 
\[
m \le 1 + \frac{d_\rho^2-1}{\dmin} \le d_\rho^2
\]
since $\rho \otimes \rho^*$ contains exactly one copy of the trivial irrep.  Even if we content ourselves with the generous bound $m \le d_\rho^2$, the largest error term in~\eqref{eq:other-terms} is $O(1)$.  Thus 
\[
\Exp_{\Pi,x} \tr \left[ \Pi \rho^\dagger(x) \Pi \rho(x) \Pi \rho^\dagger(x) \Pi \rho(x) \right] 
= d_\rho \left( \frac{d_\psi}{d_\rho} \right)^{\!3} \left( 2 - \frac{d_\psi}{d_\rho} \right) + O(1) \, .
\]
Combining this with~\eqref{eq:pirho4} gives
\[
\frob{\psi(x)^\dagger \psi(x)}^2 
= d_\psi \left( 2 - \frac{d_\psi}{d_\rho} \right) + O(1) \, ,
\]
and so~\eqref{eq:firstwenotethat} gives
\begin{equation}
\label{eq:polar-bound}
\Exp_{\Pi,x} \frob{\psi(x) - \tpsi(x)}^2 \le 
d_\psi \left( 1 - \frac{d_\psi}{d_\rho} \right) + O(1) \, .
\end{equation}

Finally, we return to our task of bounding
\[
\Exp \frob{\tpsi(xy) - \tpsi(x) \,\tpsi(y)}^2 \, .
\] 
For this purpose, we use the triangle inequality and write
\begin{align}
\frob{\tpsi(xy) - \tpsi(x) \,\tpsi(y)}  
&\le \frob{\tpsi(xy) - \psi(xy)} \label{eq:polar-single} \\
&+ \frob{\psi(xy) - \psi(x) \,\psi(y)} \label{eq:polar-hom} \\ 
&+ \frob{\psi(x) \,\psi(y) - \tpsi(x) \,\tpsi(y)} \, . \label{eq:mixed}
\end{align}
In expectation, the squares of the terms appearing in~\eqref{eq:polar-single} and~\eqref{eq:polar-hom} are precisely the topic of~\eqref{eq:polar-bound} and Theorem~\ref{thm:minor-tight}, respectively. As for the quantity~\eqref{eq:mixed}, we may further expand it as
\begin{align*}
  \frob{\psi(x) \,\psi(y) - \tpsi(x) \,\tpsi(y)} 
  &\leq \frob{\psi(x) \,\psi(y) - \tpsi(x) \,\psi(y)} + \frob{\tpsi(x) \,\psi(y) - \tpsi(x) \,\tpsi(y)} \\
  &= \frob{\big( \psi(x) - \tpsi(x) \big) \psi(y)} + \frob{\tpsi(x) \big( \psi(y)-\tpsi(y) \big)} \\
  &= \frob{\big( \psi(x) - \tpsi(x) \big) \psi(y)} + \frob{\psi(y)-\tpsi(y)} \, , 
  \end{align*}
where in the last line we used the unitarity of $\tpsi(x)$.  Squaring both sides and using the inequality $(a+b+c+d)^2 \le 4(a^2+b^2+c^2+d^2)$ gives 
\begin{align*}
\frob{\tpsi(xy) - \tpsi(x) \,\tpsi(y)}^2   
& \le 4 \left( \frob{\tpsi(xy) - \psi(xy)}^2 \right. \\
&+ \frob{\psi(xy) - \psi(x) \,\psi(y)}^2 \\
&+ \frob{\big( \psi(x) - \tpsi(x) \big) \psi(y)}^2 \\
&+ \left. \frob{\psi(y)-\tpsi(y)}^2 \right) \, . 
\end{align*}
When we take the expectation over $y$, the third term simplifies since $\psi(y)$ is unitary in expectation:
\begin{align*}
\Exp_y \frob{ \big( \psi(x) - \tpsi(x) \big) \psi(y)}^2 
&= \tr \left[ \big( \psi(x) - \tpsi(x) \big)^\dagger \left( \Exp_y \psi(y)^\dagger \psi(y) \right) \big( \psi(x) - \tpsi(x) \big) \right] \\
&= \tr \left[ \big( \psi(x) - \tpsi(x) \big)^\dagger \big( \psi(x) - \tpsi(x) \big) \right] \\
&= \frob{ \psi(x) - \tpsi(x) }^2 \, . 
\end{align*} 

Putting this together, taking expectations over $\Pi$, $x$, and $y$, using the fact that $xy$ is uniformly random, and applying~\eqref{eq:polar-bound} and Theorem~\ref{thm:minor-tight} gives
\begin{align*}
\Exp_{\Pi,x,y} \frob{\tpsi(xy) - \tpsi(x) \,\tpsi(y)}^2  
&\le 4 \left( \Exp_{x,y} \frob{\psi(xy) - \psi(x) \,\psi(y)}^2 + 3 \Exp_{\Pi,x} \frob{\psi(x) - \tpsi(x)}^2 \right) \\
&\le 2 d_\psi \left( 4 \left( 1 - \sqrt{\frac{d_\psi}{d_\rho}} \right) + 6 \left( 1-\frac{d_\psi}{d_\rho} \right) \right)
\, ,
\end{align*}
completing the proof.
\end{proof}

There are a number of ways one might improve Theorem~\ref{thm:polar}.  The bound~\eqref{eq:lambda}, and therefore~\eqref{eq:ata}, is off by a factor of $(\lambda+1)^2 \approx 4$ when $\lambda$ is close to $1$, i.e., when $\psi(x)$ is close to unitary.  Using the triangle inequality is also rather crude.  With more thought one should be able to bound $(1/d_\psi) \Exp_{\Pi,x,y} \frob{\tpsi(xy) - \tpsi(x) \,\tpsi(y)}^2$ with a smaller function of the ratio $d_\psi/d_\rho$, and thus achieve good approximate representations in lower dimensions.


\paragraph{Acknowledgments}  We benefited from the lectures of Avi Wigderson and Ben Green at the Bellairs Research Institute of McGill University, and from conversations with Jon Yard about twirled operators.  This work was supported by the NSF under grants CCF-0829931, 0835735, and 0829917, and by the NSA under contract W911NF-04-R-0009.

\end{document}